\documentclass[centering,11pt]{amsart}

\usepackage[pdftex]{graphicx} \usepackage[dvipsnames]{xcolor}

\usepackage{geometry}

\usepackage{amsfonts}
\usepackage{amsmath}
\usepackage{amssymb}
\usepackage{amsthm}
\usepackage{mathtools}
\usepackage{tikz}
\usetikzlibrary{patterns}

\usepackage{paralist}
\usepackage[colorlinks,cite color=blue,pagebackref=true,pdftex]{hyperref}

\usepackage[T1]{fontenc}

\newcommand\Def[1]{\textbf{\color{black}#1}}

\renewcommand\emptyset{\varnothing}

\newcommand\R{\mathbb{R}}

\newcommand\inner[1]{\langle {#1} \rangle}
\newcommand\defeq{\coloneqq}

\newcommand\1{\mathbf{1}}

\DeclareMathOperator{\conv}{conv}

\DeclareMathOperator{\vol}{vol}

\newtheorem{thm}{Theorem}[section] \newtheorem{cor}[thm]{Corollary}
\newtheorem{lem}[thm]{Lemma} \newtheorem{prop}[thm]{Proposition}
 \newtheorem{quest}{Question}

\theoremstyle{definition}

\title{Spectral Polyhedra}

\author{Raman Sanyal}

\address{Institut f\"ur Mathematik, Goethe-Universit\"at Frankfurt, Germany} 
\email{sanyal@math.uni-frankfurt.de}

\author{James Saunderson}

\address{Department of Electrical and Computer Systems Engineering, Monash University, VIC 3800, Australia}
\email{james.saunderson@monash.edu}



\date{\today}

\begin{document}

\begin{abstract}
    A \emph{spectral convex set} is a collection of symmetric matrices whose
    range of eigenvalues form a symmetric convex set.  Spectral convex sets
    generalize the Schur-Horn orbitopes studied by Sanyal--Sottile--Sturmfels
    (2011). We study this class of convex bodies, which is 
    closed under intersections, polarity, and Minkowski sums. We describe
    orbits of faces and give a formula for their Steiner polynomials.  We then
    focus on spectral polyhedra. We prove that spectral polyhedra are
    spectrahedra and give small representations as spectrahedral shadows. We
    close with observations and questions regarding hyperbolicity cones, polar
    convex bodies, and spectral zonotopes.
\end{abstract}
\maketitle

\newcommand\SymGrp{\mathfrak{S}}%
\newcommand\Sym{\mathrm{S}_2}%
\newcommand\SymMat{\mathrm{S}_2\R^d}%
\newcommand\Spec[1]{\Lambda(#1)}%

\section{Introduction}\label{sec:intro}
The symmetric group $\SymGrp_d$ acts on $\R^d$ by permuting coordinates. We
call a convex set $K \subset \R^d$ \Def{symmetric} if $\sigma K = K$ for all
$\sigma \in \SymGrp_d$. We write $\SymMat$ for the
$\binom{d+1}{2}$-dimensional real vector space of
symmetric $d$-by-$d$ matrices. Every real symmetric matrix $A \in \SymMat$ has
$d$ real eigenvalues, which we denote by $\lambda(A) \in \R^d$.  In this note,
we are concerned with sets of the form
\begin{equation}\label{eqn:Def_LK}
    \Spec{K} \ \defeq \ \{ A \in \SymMat : \lambda(A) \in K \} \, ,
\end{equation}
which we call \Def{spectral convex sets}. 
\newcommand\diag{\delta}%
The name is justified by Corollary~\ref{cor:convex} which asserts that $\Spec{K}$ is indeed a convex
subset of $\SymMat$. 

\newcommand\SH{\mathcal{SH}}%
The simplest symmetric convex sets are of the form $\Pi(p) = \conv\{ \sigma p
: \sigma \in \SymGrp_d \}$ for $p \in \R^d$. Such a symmetric polytope is
called a \Def{permutahedron}~\cite{BilleraSarangarajan} and the associated spectral convex sets $\SH(p)
:= \Spec{\Pi(p)}$ were studied in~\cite{SSS11} under the name
\Def{Schur-Horn orbitopes}. The class of spectral convex sets is strictly
larger. For example, for $1\leq p \leq \infty$, the unit $p$-norm ball in $\R^d$ is a 
symmetric convex set. The associated spectral convex set 
is the unit \Def{Schatten $p$-norm ball} in $\SymMat$, consisting of $d\times d$ symmetric matrices with
eigenvalues having $p$-norm at most one. It follows that the spectral convex set 
associated with the cube in $\R^d$ is the spectral norm ball in $\SymMat$, the spectral 
convex set associated with the octahedron in $\R^d$ is the nuclear norm ball in $\SymMat$, and 
the spectral convex set associated with the Euclidean norm ball is the 
Frobenius norm ball. 

In Section~\ref{sec:spectral_bodies}, we summarize some basic, yet remarkable, geometric and algebraic
properties of spectral convex sets. In particular, we observe that spectral
convex sets are closed under intersections, Minkowski sums, and polarity.

A \Def{spectrahedron} is a convex set $S \subset \R^d$ of the form
\[
    S \ = \ \{ x \in \R^d : A_0 + x_1 A_1 + \cdots + x_d A_d \succeq 0 \} \, ,
\] 
where $A_0,A_1,\ldots,A_d$ are symmetric matrices and $\succeq$ denotes
positive semidefiniteness.  In Section~\ref{sec:spectrahedra}, we show that
\Def{spectral polyhedra}, that is, spectral convex bodies associated to
symmetric polyhedra, are spectrahedra (Theorem~\ref{thm:spectrahedron}),
generalizing the construction from~\cite{SSS11} for Schur-Horn orbitopes. It
follows that spectral polyhedra are basic semialgebraic, and are examples of
the very special class of \emph{doubly spectrahedral} convex sets, i.e.,
spectrahedra whose polars are also spectrahedra~\cite{SPW15}.  Spectral
polyhedral cones are hyperbolicity cones (see Section~\ref{sec:hyp} for
details). The generalized Lax conjecture asserts that every hyperbolicity cone
is spectrahedral. Theorem~\ref{thm:spectrahedron}, therefore, gives further
positive evidence for the generalized Lax conjecture.

If $P$ is a symmetric polyhedron with $M$ orbits of defining inequalities,
then the size of our spectrahedral representation of $\Spec{P}$ is $M \cdot
\prod_{i=1}^d \binom{d}{i}$. A lower bound on the size of a spectrahedral
representation is $M d{!}$, obtained by considering the degree of the
algebraic boundary.  While spectrahedral representations give insight into the
algebraic properties of spectral polyhedra, in order to solve convex
optimization problems involving spectral polyhedra, it suffices to give
representations as \Def{spectrahedral shadows}, i.e., linear projections of
spectrahedra.  In Section~\ref{sec:shadows}, we use a result of Ben-Tal and
Nemirovski~\cite{BenTal-Nemirovski} to give significantly smaller
representations of spectral polyhedra as spectrahedral shadows. 

We close in Section~\ref{sec:misc} with remarks, questions, and future
directions regarding hyperbolic polynomials and the generalized Lax
conjecture, generalizations to other Lie groups, and spectral zonotopes.

\textbf{Acknowledgements.} The first author thanks Oliver Goertsches, Leif
Nauendorf, Luke Oeding, Thomas Wannerer, and Anna-Laura Sattelberger for
insightful conversations. This project was initiated while the first author
was visiting the Mathematical Sciences Research Institute (MSRI) and the
second author was visiting the Simons Institute for the Theory of Computing.
We would like to thank the organizers of the programs \emph{Geometric and
Topological Combinatorics} and \emph{Bridging Continuous and Discrete
Optimization} for creating a stimulating atmosphere and encouraging
interaction.

\section{Spectral convex sets}
\label{sec:spectral_bodies}

\newcommand{\Diag}{D}
Denote by $\Diag : \SymMat \to \R^d$ the projection onto the diagonal and by
$\diag : \R^d \to \SymMat$ the embedding into diagonal matrices.  Many
remarkable properties of spectral convex sets arise because the projection
onto the diagonal, and the diagonal section, coincide.

\begin{lem}\label{lem:sec_project}
    If $K$ is a symmetric convex set, then
    \[
        \Diag(\Spec{K}) \ = \ K \ = \ \Diag( \Spec{K} \cap \diag(\R^d) ).
    \]
\end{lem}

\newcommand\major{\trianglelefteq}%

Before giving a proof, we introduce some notation and terminology. 
For a point $p \in \R^d$, we write $s_k(p)$ for the sum of its $k$ largest
coordinates. Recall that a point $q \in \R^d$ is \Def{majorized} by $p$, denoted $q \major p$, if
\begin{equation}
    \sum_{i=1}^d q_i \ = \ \sum_{i=1}^d p_i \qquad \text{ and } \qquad s_k(q)
    \ \le \ s_k(p) \quad \text{ for all } k=1,\dots,d-1 \, .
\end{equation}
Majorization relates to permutahedra in that
\[ 
    \Pi(p) \ = \  \{q\in \R^d\;:\; q \major p\}.
\]
In other words, the majorization inequalities give an inequality description
of the permutahedron~\cite{BilleraSarangarajan}.

\begin{proof}[{Proof of Lemma~\ref{lem:sec_project}}]
    Since $\Spec{K}$ contains $\diag(K)$, the obvious inclusions are
	$\Diag(\Spec{K}) \supseteq K$ and $K \subseteq \Diag(\Spec{K} \cap \diag(\R^d))$.

    For the remaining inclusions, we use Schur's insight (see, for
    example,~\cite[Thm.~4.3.45]{HornJohnson}) that
    for any $A \in \SymMat$, we have $\Diag(A) \major \lambda(A)$. Since
    $\Pi(p) \subseteq K$ for any $p \in K$, we infer that $\Diag(A) \in K$ for any
    $A \in \Spec{K}$.
\end{proof}

Lemma~\ref{lem:sec_project} yields that spectral convex sets are, in fact, convex. 

\begin{cor}\label{cor:convex}
    If $K$ is a symmetric convex set, then $\Spec{K}$ is convex.
\end{cor}
\begin{proof}
	It is enough to show that $\conv(\Spec{K}) \subseteq \Spec{K}$. 
	Assume that $A \in \conv( \Spec{K} )$. We can assume
    that $A = \diag(p)$ for some $p \in \R^d$. By
    definition there are $A_1,\dots,A_m \in \Spec{K}$ such that $\diag(p) =
    \sum_{i=1}^m \mu_i A_i$ with $\mu_i \ge 0$ and $\mu_1+\cdots+\mu_m = 1$.
    In particular, $p = \Diag(A) = \sum_i \mu_i \Diag(A_i)$ and 
    Lemma~\ref{lem:sec_project} yields $p \in K$. It follows that $A\in \Spec{K}$. 
\end{proof}

\newcommand\tr{\mathop{tr}}%
\newcommand{\support}[2]{h_{#1}(#2)}%
We identify the dual space $(\SymMat)^*$ with $\SymMat$ via the Frobenius
inner product $\inner{A,B} \defeq \tr( A B)$. The \Def{support function} of a
closed convex set $K$ is defined by 
\[
	\support{K}{c} \ \defeq \ \max \{ \inner{c,p} : p \in K \} \, .
\]

\begin{prop}\label{prop:support}
    If $K \subset \R^d$ is a symmetric closed convex set, then
	$ \support{\Spec{K}}{B} \ = \ \support{K}{\lambda(B)}$
    for all $B \in \SymMat$.
\end{prop}
\begin{proof}
    Let $B = g B' g^t$ for $g \in O(d)$ and $B'$ diagonal.  Using the fact
    that the trace is invariant under cyclic shifts, we see that
	$\support{\Spec{K}}{B} = \support{\Spec{K}}{B'}$. Lemma~\ref{lem:sec_project} and the fact
    that $\inner{A,B'} = \inner{\Diag(A),\Diag(B')}$ finishes the proof.
\end{proof}

\newcommand\OF{\overline{\mathcal{F}}}%
Faces of $\Spec{K}$ and $K$ come in $O(d)$- and $\SymGrp_d$-orbits,
respectively. The collection of faces up to symmetry is a partially ordered
set with respect to inclusion that we denote by $\OF(\Spec{K})$ and $\OF(K)$
respectively.

\begin{cor}
    For any symmetric convex body $K \subset \R^d$, the posets $\OF(K)$ and
    $\OF(\Spec{K})$ are canonically isomorphic.
\end{cor}

\newcommand\polar[1]{{#1}^\circ}%
The \Def{polar} of a convex set $K \subset \R^d$ is defined as
\[
	\polar{K} \ \defeq \ \{ y \in \R^d : \support{K}{y} \le 1 \} \, .
\]
It is easy to see that the polar of a symmetric convex set is symmetric.
In combination with Proposition~\ref{prop:support}, we can deduce that the class
of spectral convex sets is closed under polarity. 

\begin{thm} \label{thm:polarity}
    If $K$ is a closed symmetric convex set, then $\polar{\Spec{K}} \ = \ \Spec{\polar{K}}$.
\end{thm}
\begin{proof}
    For $B \in \SymMat$, we have $B \in \polar{\Spec{K}}$ if and only if $1
	\ge \support{\Spec{K}}{B} = \support{K}{\lambda(B)}$, which happens if and only if
    $\lambda(B) \in \polar{K}$.
\end{proof}
Furthermore, since polyhedra are also closed under polarity, it follows that
the class of spectral polyhedra is closed under polarity.
Theorem~\ref{thm:polarity}, like many of the convex analytic facts in this
section, can be deduced from results of Lewis on extended real-valued spectral
functions~\cite{lewis1996convex}.

Proposition~\ref{prop:support} can also be used to show that spectral convex
bodies interact nicely with Minkowski sums.

\begin{cor} \label{cor:Minkowski_sum}
    If $K,L \subset \R^d$ are symmetric convex bodies, then
        $\Spec{K+L}  =  \Spec{K} + \Spec{L}$.
\end{cor}
\begin{proof}
    We compute
    \begin{align*}
	    \support{\Spec{K}+\Spec{L}}{B} & \ = \    
	\support{\Spec{K}}{B}  + \support{\Spec{L}}{B}  \ = \    
	\support{K}{\lambda(B)}  + \support{L}{\lambda(B)}\\ &  \ = \    
	\support{K+L}{\lambda(B)}  \ = \   
	\support{\Spec{K+L}}{B} \, .
    \qedhere
    \end{align*}
\end{proof}

We can use this property to  simplify the computation of basic
convex-geometric invariants; cf.~the book by Schneider~\cite{Schneider}. Let $B(\R^d)$ denote the Euclidean unit ball in
$\R^d$. The \Def{Steiner polynomial} of a convex body $K \subset \R^d$ is 
\[
    \vol(K + t B(\R^d)) \ = \ 
    W_d(K) + d W_{d-1}(K) t + \cdots + \tbinom{d}{d} W_0(K) t^d \, .
\]
The coefficients $W_i(K)$ are called \Def{querma{\ss}integrals}.  The
following reduces the computation of Steiner polynomials of $\Spec{K}$ to the
computation of an integral over $K$.
\begin{thm}
    Let $K \subset \R^d$ be a symmetric convex body. Then
    \[
    \vol(\Spec{K} + t B(\SymMat)) \ = \
    2^{\frac{1}{2}d(d+3)}\prod_{r=1}^d\frac{\pi^{\frac{r}{2}}}{\Gamma(\frac r 2)}
    \int_{K + t B_d}
    \prod_{i < j} |p_j - p_i| \,  dp
    \]
\end{thm}
\begin{proof}
    Recall from the introduction that the unit ball in $\SymMat$ satisfies
    $B(\SymMat)  = \Spec{B(\R^d)}$. In particular, using
    Corollary~\ref{cor:Minkowski_sum}, we need to determine the volume of
    $\Spec{K + t B(\R^d)}$. 

    Let $\varphi : O(d) \times \R^d \to \SymMat$ with $\varphi(g,p) := g
    \diag(p) g^t$. Then by Corollary~\ref{cor:convex}, we need to compute
    $\int_{\varphi(O(d) \times K')} d\mu$, where $K' := K + t B(\R^d)$.

    The differential at $(g,p) \in O(d) \times \R^d$ is the linear map
    $D_{g,p} : T_g O(d) \times T_p \R^d \to T_{\varphi(g,p)} \SymMat$ with
    \[
        D_{g,p}\varphi (Bg, u) \ = \ [g\delta(p)g^t,B] + g D(u) g^t \, ,
    \]
    where $[\,,]$ is the Lie bracket. Now, the linear spaces $T_g O(d) \times
    T_p \R^d$ and  $T_{\varphi(g,p)} \SymMat$ have the same dimension. If $g =
    (g_1,g_2,\dots,g_d) \in O(d)$, then we choose as a basis for the former
    $g_i \wedge g_j := g_i g_j^t - g_j g_i^t \in T_gO(d)$ for $1 \le i < j \le
    d$ and the standard basis $e_1,\dots,e_d \in T_p\R^d = \R^d$. For the
    latter, we choose $g_i \bullet g_j = \frac{1}{2}(g_ig_j^t + g_jg_i^t)$ for
    $1 \le i < j \le d$ and $g_i \bullet g_i $ for $i=1,\dots,d$. We then
    compute
    \[
        D_{g,p}(g_i \wedge g_j) \ = \ (p_j - p_i) \, g_i \bullet g_j
        \quad \text{and} \quad 
        D_{g,p}(e_i) \ = \ g_i \bullet g_i \, .
    \]
    Hence, under the identification $g_i \wedge g_j \mapsto g_i \bullet g_j$
    and $e_i \mapsto g_i \bullet g_i$, $D_{g,p}\varphi$ has eigenvalues $p_j -
    p_i$ for $i < j$ as well as $1$ with multiplicity $d$. This yields
    \[
        \int_{\varphi(O(d) \times K'} \, d\mu \ = \ \int_{O(d) \times K'}
        |\det D_{g,p}\varphi| \, dg dp \ = \  \int_{O(d)}dg \, \int_{K'} \prod_{i <
        j} |p_j - p_i| \, dp \, .
    \]
    Together with Hurwitz formula for the volume of $O(d)$, this yields the
    claim.
\end{proof}

\newcommand\apartial{\partial_\mathrm{alg}}%
The \Def{algebraic boundary} $\apartial K$ of a full-dimensional
closed convex set $K \subset \R^d$ is, up to scaling, the unique polynomial
$f_K \in \R[x_1,\dots,x_d]$ of minimal degree that vanishes on all points $q
\in \partial K$.  If $K$ is symmetric, then $f_K$ is a symmetric polynomial,
that is, $f_K(x_{\sigma^{-1}(1)},\dots,x_{\sigma^{-1}(d)}) =
f_K(x_1,\dots,x_d)$ for all $\sigma \in \SymGrp_d$. By the fundamental theorem
of symmetric polynomials, there is a polynomial $F_K(y_1,\dots,y_d) \in
\R[y_1,\dots,y_d]$ such that $f_K(x_1,\dots,x_d) = F_K(e_1,\dots,e_d)$, where
$e_i$ is the $i$-th elementary symmetric polynomial.

For $A \in \SymMat$, let $\det(A + tI) =  t^d + \eta_1(A) t^{d-1} + \cdots +
\eta_d(A)$ be its characteristic polynomial. The coefficients $\eta_i(A)$ are
polynomials in the entries of $A$ and it is easy to see that $\eta_i(g A g^t)
= \eta_i(A)$. In fact, every polynomial $h$ such that $h(gAg^t) = h(A)$ for
all $g \in O(d)$ and $A \in \SymMat$ can be written as a polynomial in
$\eta_1,\dots,\eta_d$; see \cite[Ch.~12.5.3]{GW}. 

\begin{prop} \label{prop:alg}
    Let $K \subset \R^d$ be a symmetric closed convex set. Then the algebraic
    boundary of $\Spec{K}$ is given by $F_K(\eta_1,\dots,\eta_d)$. In
    particular, $\apartial K$ and $\apartial \Spec{K}$ have the same degree.
\end{prop}
\begin{proof}
    The first part follows from the discussion above. For the second part, we
    simply note that the collection of polynomials $e_i$ and $\eta_i$ are
    algebraically independent with corresponding degrees.
\end{proof}

\section{Spectrahedra}\label{sec:spectrahedra}

In this section, we show that spectral polyhedra are spectrahedra. For $P =
\Pi(p)$ a permutahedron and $\SH(p) = \Spec{P}$, a Schur-Horn orbitope, this was
shown in~\cite{SSS11}. We briefly recall the construction, which will then be
suitably generalized.

A point $q \in \R^d$ is contained in $\Pi(p)$ if and only if $q \major p$.
This condition can be rewritten in terms of linear inequalities. For $I
\subseteq [d]$, we write $q(I) = \sum_{i \in I} q_i$. Then $q \major p$ if and
only if
\[
    s_d(p) \ = \ q([d]) \quad \text{ and } \quad
    s_{|I|}(p) \ \ge \ q(I) \quad \text{ for all } \emptyset \neq I
    \subsetneq [d].
\]
If $p$ is generic, that is, $p_i \neq p_j$ for $i \neq j$, then it is easy to
show that the system of  $2^d - 2$ linear inequalities is irredundant.

\newcommand\SF{\mathcal{L}}%
For $1 \le k \le d$, the \Def{$k$-th linearized Schur functor} $\SF_k$ is a
linear map from $\SymMat$ to $\Sym\bigwedge^k \R^d$ such that the eigenvalues
of $\SF(A)$ are precisely $\lambda(A)(I) = \sum_{i \in I} \lambda(A)_i$ for $I
\subseteq [d]$ and $|I| = k$.  Therefore, $\SH(p)$ is precisely the set
of points $A \in \SymMat$ such that 
\begin{equation}\label{eqn:SH_spec}
    s_d(p) \ = \ \tr(A) \quad \text{ and } \quad
    s_k(p) \, I_{\binom{d}{k}} \ \succeq \ \SF_k(A) \quad \text{ for all }
    1 \le k < d \, .
\end{equation}

The simplest symmetric polyhedron has the form
\[ 
    P_{a,b} \ = \  \{x\in \R^d : \inner{ \sigma a, x} \le b \text{ for }
    \sigma \in \SymGrp_d\}%
\]
where $a\in \R^d$ and $b\in \R$. In general, a symmetric polyhedron has the form
\[
    P \ = \ \{ x \in \R^d : \inner{ \sigma a_i, x } \le b_i \text{ for } \sigma
    \in \SymGrp_d \text{ and } i = 1,\dots,M \} 
    \ = \ 
    \bigcap_{i=1}^M P_{a_i,b_i}
    \,, 
\]
Since $\Spec{K \cap L} \ = \ \Spec{K} \cap \Spec{L}$, it suffices to focus on the case
$P_{a,b}$.

To extend the representation~\eqref{eqn:SH_spec} directly, for each general $a\in \R^d$, we would need a
linear map $\SF_a$ from $\SymMat$ to $\Sym V$ with $\dim V = d!$ such that the
eigenvalues of $\SF_a(A)$ are precisely $\inner{\sigma a, \lambda(A)}$ for all
$\sigma \in \SymGrp_d$.  For $a = (1,\dots,1,0,\dots,0)$ with $k$ ones, this
is realized by the linearized Schur functors. 

\newcommand\adj{\mathop{adj}}%
\begin{prop}
    For $d = 2$, set
    \[
        \SF_a(A) \ := \ a_1 A + a_2 \adj(A) \, ,
    \]
    where $\adj(A)$ is the adjugate (or cofactor) matrix. Then $A
    \mapsto \SF_a(A)$ is a linear map satisfying the above requirements.
\end{prop}
\begin{proof}
    Since $d=2$, the map $A \mapsto \adj(A)$ is linear.  The matrices $A$ and
    $\adj(A)$ can be simultaneously diagonalized and hence it suffices to
    assume that $A = \diag(\lambda_1,\lambda_2)$. In that case $\adj(A) =
    \diag(\lambda_2,\lambda_1)$, which proves the claim.
\end{proof}

The construction above only works for $d=2$ and we have not been able to
construct such a map for $d \ge 3$.

\begin{quest}
    Does $\SF_a$ exist for $d \ge 3$?
\end{quest}

\newcommand\cI{\mathcal{I}}%
We pursue a different approach towards a spectrahedral representation by
considering a redundant set of linear inequalities for $P_{a,b}$. An
ordered collection $\cI = (I_1,\dots,I_d)$ of subsets $I_j \subseteq [d]$ is
called a \Def{numerical chain} if $|I_j| = j$ for all $j$. A numerical chain is
a \Def{chain} if additionally $I_1 \subset I_2 \subset \cdots \subset I_d$.
Chains are in bijection to permutations $\sigma \in \SymGrp_d$ via $I_j = \{
\sigma(1),\dots,\sigma(j) \}$. For $I \subseteq [d]$, we write $\1_I \in
\{0,1\}^d$ for its characteristic vector.

Let us assume that $a = (a_1 \ge a_2 \ge \cdots \ge a_d)$ and set $a_{d+1} :=
0$. For a numerical chain $\cI$, we define
\begin{equation}\label{eqn:cI} 
    a^\cI \ := \  (a_1 - a_2)\1_{I_1} + (a_2 - a_3)\1_{I_2} + \cdots +
    (a_{d-1} - a_d)\1_{I_{d-1}} + a_d \1_{I_d} \, .
\end{equation}
    
\begin{prop}\label{prop:Pi_redundant}
    Let $a = (a_1 \ge a_2 \ge \cdots \ge a_d)$ and $b\in \R$. Then 
    \[
	    P_{a,b}\ = \ \{ x \in \R^d : \inner{a^\cI,x} \le b \text{ for all
    numerical chains } \cI \} \, .
    \]
\end{prop}
\begin{proof}
    Let $Q$ denote the right-hand side. To see that $Q \subseteq
	P_{a,b}$, %
	we note that if $\cI$ is a chain corresponding to a
    permutation $\sigma$, then $a^\cI = \sigma a$.

    For the reverse inclusion, it suffices to show that $a^\cI \major a$,
    which implies that  $\inner{a^\cI,x} \le b$ is a valid inequality for
	$P_{a,b}$. %
	Using the fact that $s_k(p+q) \le s_k(p) + s_k(q)$, we
    compute
    \[
        s_k(a^\cI)  \le  \sum_{j=1}^d (a_j - a_{j+1}) s_k(\1_{I_j})
         =  \sum_{j=1}^{k-1} j (a_j - a_{j+1}) + k \sum_{j=k}^d (a_j -
        a_{j+1})  =  a_1 + \cdots + a_k  =  s_k(a) \, .
    \]
    Similarly $s_d(a^\cI) = a_1 + \cdots + a_d$, which completes the
    proof.
\end{proof}

\newcommand\SFh{\widehat{\SF}}%

Recall that for matrices $A \in \SymMat$ and $B \in \Sym\R^e$, the tensor
product $A \otimes B$ is a symmetric matrix of order $de$ with eigenvalues
$\lambda_i(A) \cdot \lambda_j(B)$ for $i=1,\dots,d$ and $j=1,\dots,e$. For $a
= (a_1 \ge \cdots \ge a_d)$, let 
\[
    \SFh_a : 
    \textstyle{\bigwedge^1\R^d} \otimes
    \bigwedge^2\R^d \otimes
    \cdots \otimes
    \bigwedge^d\R^d
    \ \to \
    \bigwedge^1\R^d \otimes
    \bigwedge^2\R^d \otimes
    \cdots \otimes
    \bigwedge^d\R^d
\]
be the linear map given by 
\[
    \SFh_a(A) \ \defeq \ \sum_{j=1}^d (a_j - a_{j+1}) 
    I_{\binom{d}{1}} \otimes
    \cdots \otimes
    I_{\binom{d}{j-1}} \otimes
    \SF_j(A) \otimes
    I_{\binom{d}{j+1}} \otimes
    \cdots \otimes
    I_{\binom{d}{d}}  \, .
\]

\begin{thm}\label{thm:spectrahedron}
	Let $P = P_{a_1,b_1} \cap \cdots \cap P_{a_M,b_M}$
	be a symmetric polyhedron. Then $A \in \Spec{P}$ if and only if
	\[ b_i\, I \succeq \SFh_{a_i}(A)\;\;\textup{for $i=1,2,\ldots,M$}.\]
\end{thm}
\begin{proof}
Since $\Spec{P} = \bigcap_{i=1}^{M}\Spec{P_{a_i,b_i}}$ it is enough to show that 
	$A\in \Spec{P_{a,b}}$ if and only if $bI \succeq \SFh_{a}(A)$. 

    Let $a = (a_{1} \ge a_{2} \ge \cdots \ge a_{d})$ and $A \in \SymMat$
    with $v_1,\dots,v_d$ an orthonormal basis of eigenvectors. For $I = \{i_1
    < i_2 < \cdots < i_k \}$ a subset of $[d]$, we write $v_I := v_{i_1}
    \wedge v_{i_2} \wedge \cdots \wedge v_{i_k} \in \bigwedge^k \R^d$. Then a
    basis of eigenvectors for $\SFh_a(A)$ is given by
    \[
        v_\cI \ := \ v_{I_1} \otimes v_{I_2} \otimes \cdots \otimes v_{I_d} \,
        ,
    \]
    where $\cI$ ranges of all numerical chains. The eigenvalue of $\SFh_a(A)$
    corresponding to $v_\cI$ is precisely $\inner{a^\cI,\lambda(A)}$. Hence
    $A$ satisfies the given linear matrix inequalities for $a$ if and only
    if $\sum_i \lambda_i(A) = \sum_i a_i$ and $\inner{a^\cI,\lambda(A)} \le b$
    for all $\cI$. By Proposition~\ref{prop:Pi_redundant} this is the case if
	and only if $\lambda(A) \in P_{a,b}$ or, equivalently, $A \in
	\Spec{P_{a,b}}$.
\end{proof}

The spectrahedral representation given in Theorem~\ref{thm:spectrahedron} for
$\Spec{P}$, where $P$ is a symmetric polyhedron in $\R^d$ with $M$ orbits of
facets, is of size
\[
    M \cdot \prod_{i=1}^d \binom{d}{i}.
\]
So the spectrahedral representation is of order $M 2^{d^2}$; see~\cite{LagariasMehta}.

If 
\[
    K \ = \ \{x\in \R^d : A_0 + x_1 A_1 + \cdots + x_d A_d \succeq 0\}
\]
is a spectrahedral representation of a convex set $K$ with $A_0,\dots,A_d \in
\Sym \R^m$ and $A_0$ positive definite, then $h(x) = \det(A_0 + x_1 A_1 +
\cdots + x_d A_d)$ vanishes on $\partial K$. Hence, the size of a
spectrahedral representation is bounded from below by the degree of $\apartial
K$. If $P$ is a symmetric polytope with $M$ full orbits of facets, then its
algebraic boundary has degree $M\cdot d{!}$. From the discussion following
Proposition~\ref{prop:alg}, we can deduce that the degree of $\apartial
\Spec{P}$ is also $M\cdot d{!}$, and so that any spectrahedral representation
of $\Spec{P}$ has size at least $M\cdot d{!}$. While interesting from an
algebraic point of view, spectrahedral representations of symmetric polytopes
are clearly impractical for computational use. In the next section we discuss
substantially smaller representations as projections of spectrahedra.

\section{Spectrahedral shadows}\label{sec:shadows}

In this section, we give a representation of $\Spec{K}$ as a spectrahedral shadow, i.e., 
a linear projection of a spectrahedron, when $K$ is, itself, a symmetric spectrahedral shadow, 
by a direct application of results from~\cite{BenTal-Nemirovski}. The aim of this section is to 
illustrate the significant reductions in size possible by using projected spectrahedral representations. 

It is convenient to use slightly different notation in this section, to
emphasize that we do not need to construct an explicit representation of the
symmetric convex set $K$, to get a representation of $\Spec{K}$.  To this end,
let $\R^d_{\downarrow} = \{p\in \R^d\;:\; p_1 \geq p_2 \geq \cdots \geq
p_d\}$.  For $L\subseteq \R^d_{\downarrow}$ define \[ \Pi(L) =
\conv\,\left(\SymGrp_d \cdot L\right),\] the convex hull of the orbit of $L$
under $\SymGrp_d$.  This is the inclusion-wise minimal symmetric convex set
containing $L$.  We recover the usual permutahedron of a point $p\in
\R^d_{\downarrow}$ by $\Pi(p)$. 

In Theorem~\ref{thm:projspec}, we give a representation of $\Spec{\Pi(L)}$ as a 
spectrahedral shadow whenever $L\subseteq \R_{\downarrow}^d$ is a spectrahedral 
shadow. We use the following result of Ben-Tal and Nemirovski~\cite[Section~4.2,
18c]{BenTal-Nemirovski}.

\begin{lem}\label{lem:BTN}
    Let $1 < k < d$ and $t \in \R$. Then a matrix $A \in \SymMat$ satisfies
    $s_k(\lambda(A)) \le t$ if and only if there are $Z \in \SymMat$ and $s
    \in \R$ such that
    \[
    Z \ \succeq \ 0, \quad Z - A + s I_d \ \succeq \ 0, \quad \text{ and }
    \quad t - ks - \tr(Z) \ \ge \ 0 \, .
    \]
\end{lem}

For the case $k = 1$, we obtain the simpler representation $s_1(\lambda(A)) =
\max \lambda(A) \le t$ if and only if $t I - A \succeq 0$.

\begin{thm}
	\label{thm:projspec}
	If $L\subseteq \R^d_{\downarrow}$ is convex then 
	\begin{equation}
		\label{eq:specpil}
		\Spec{\Pi(L)} = \{A\in \SymMat\;:\; \exists p\in L\;\textup{such that}\;\lambda(A) \major p\}\,.
	\end{equation}
	If $L\subseteq \R^d_{\downarrow}$ is the projection of a spectrahedron of size $r$, then 
	 $\Spec{\Pi(L)}$ is the projection of a spectrahedron of size $r +  2 d^2 - 2d - 2$.
\end{thm}
\begin{proof}
    Let $C$ denote the right-hand side of~\eqref{eq:specpil}. We first show
    that $C$ is convex and is the projection of a spectrahedron of size
    $r+2d^2-2d-2$. Since $p\in L\subseteq \R_{\downarrow}^d$, we can write
    $s_k(p) = \sum_{i=1}^{k}p_i$, which is linear in $p$.  Then, using
    Lemma~\ref{lem:BTN}, the conditions $\tr(A) = \sum_i p_i$ and
    $s_k(\lambda(A)) \le \sum_{i=1}^k p_i$ for $1 \le k \le d-1$ define a
    convex set in $A$ and $p$.  Moreover, this set can be encoded by linear
    matrix inequalities involving matrices of size $(d-2)(2d+1) + d$, for a
    total size of $r + (d-2)(2d+1) + d  = r + 2d^2 - 2d - 2$.
	
    To check that $\Spec{\Pi(L)} = C$, since both sides are spectral convex
    sets, it is enough to check that their diagonal projections are equal.
    Since $\Pi(L)$ is symmetric, $\Diag(\Pi(L)) = \Pi(L)$. The diagonal
    projection $\Diag(C)$ is a symmetric convex set containing $L$, so
    $\Diag(C) \supseteq \Pi(L)$. For the reverse inclusion, if $A\in C$ then
    there exists $p\in L$ such that $\lambda(A) \major p$, but then $A\in
    \Spec{\Pi(p)}\subseteq \Spec{\Pi(L)}$. 
\end{proof}
We now specialize to the case of $\Spec{P}$ where $P$ is a symmetric
polyhedron with the origin in its interior. 
\begin{prop}
    Suppose that $P\subseteq \R^d$ is a symmetric polyhedron with $M$ orbits
    of facets that contains the origin in its interior.  Then $\Spec{P}$ is
    the projection of a spectrahedron of size $M + 2d^2 - 2d-2$. 
\end{prop}
\begin{proof}
    We will argue that $\Spec{\polar{P}} = \polar{\Spec{P}}$ is the projection
    of a spectrahedron of size $M+2d^2-2d-2$, and then appeal to the fact that
    if $C$ has a projected spectrahedral representation then $\polar{C}$ has a
    representation of the same size~\cite[Proposition 1]{GPT13}.  By our
    assumptions on $P$, we have that $\polar{(\polar{\Spec{P}})} = \Spec{P}$. 

    Since the origin is in the interior of $P$, we know that $\polar{P}$ is a
    symmetric polytope with $M$ orbits of vertices. Each orbit of
    vertices meets $\R^d_{\downarrow}$ and thus $\Spec{P} =
    \Spec{\Pi(\{v_1,\ldots,v_M\})}$ for some $v_1,\ldots,v_M\in
    \R^d_{\downarrow}$.  Let $L = \conv\,\{v_1,\ldots,v_M\}\subseteq
    \R^d_{\downarrow}$, and note that 
    \[
        L = \{ \mu_1v_1 + \cdots + \mu_M v_M\;:\; \mu_1,\dots,\mu_M \ge
        0,\;\mu_1 + \cdots + \mu_M= 1\}
    \] 
    gives a representation of $L$ as the projection of a polyhedron with $M$
    facets, and so a projected spectrahedral representation of size $M$.
    Finally, since $\Pi(L) = \Pi(\{v_1,\ldots,v_M\})$, it follows from
    Theorem~\ref{thm:projspec} applied to $\Spec{\Pi(L)}$ that
    $\polar{\Spec{P}} = \Spec{\polar{P}}$ is the projection of a spectrahedron
    of size $M+2d^2-2d-2$. 
\end{proof}

\section{Remarks, Questions, and future directions}\label{sec:misc}

\subsection*{Hyperbolicity cones and the generalized Lax conjecture}
\label{sec:hyp}
A multivariate polynomial $f \in \R[x_1,\dots, x_d]$, homogeneous of degree
$m$, is \Def{hyperbolic} with respect to  $e\in \R^d$ if $f(e) \neq 0$ and for
each $x\in \R^d$, the univariate polynomial $t\mapsto f_x(t) := f(x-te)$ has
only real roots. Associated with $(f,e)$ is a closed convex cone $C_{f,e}
\subseteq \R^d$, defined as the set of points $x \in \R^d$ for which all roots
of $f_x$ are non-negative. A major question in convex algebraic geometry,
known as the \emph{generalized (set-theoretic) Lax conjecture}
(see~\cite{Vinnikov}), asks whether every hyperbolicity cone is a
spectrahedron. 

If $C = \{x\in \R^d: \inner{\sigma a_i,x} \geq 0,\;\textup{for all $\sigma \in
\SymGrp_d$ and $i=1,2,\ldots, M$}\}$ is a symmetric polyhedral cone
containing $e=(1,1,\ldots,1)$ in its interior, then it is the hyperbolicity
cone associated with the degree $M\cdot d{!}$ symmetric polynomial  
\[ 
    f(x) = \prod_{i=1}^{M}\prod_{\sigma\in \SymGrp_d}\inner{\sigma a_i,x}.
\] 
The spectral polyhedral cone $\Spec{C}$ is the hyperbolicity cone associated
with the polynomial $F(X) = f(\lambda(X))$ and $e = I \in \SymMat$. This
follows from Proposition~\ref{prop:alg} and is a special case of an
observation of Bauschke, G\"uler, Lewis, and Sendov~\cite[Theorem
3.1]{BGLS01}. One can view Theorem~\ref{thm:spectrahedron} as providing
further evidence for the generalized Lax conjecture, since it shows that every
member of this family of hyperbolicity cones is, in fact, a spectrahedron.

Given a symmetric hyperbolic polynomial $f$, one natural way to produce a new
symmetric hyperbolic polynomial, and an associated symmetric hyperbolicity
cone, is to take the directional derivative $D_ef$ in the direction
$e=(1,1,\ldots,1)$, an example of a \emph{Renegar derivative}.  This operation
commutes with passing to the associated spectral objects. Indeed, taking the
Renegar derivative $D_ef$ and then constructing the spectral convex cone
$\Lambda(C_{D_ef,e})$ gives the same result as constructing the spectral
hyperbolic polynomial $F(X) = f(\lambda(X))$ and then taking the hyperbolicity
cone of $D_IF$, the Renegar derivative in the direction $I\in \SymMat$. For
example, the hyperbolicity cones associated with the elementary symmetric
polynomials are symmetric convex cones that arise by repeatedly taking Renegar
derivatives starting with $f(x) = x_1 x_2 \cdots x_d$ in the direction
$e=(1,1,\ldots,1)$.  Br\"and\'en~\cite{branden2014hyperbolicity} established
that these cones are all spectrahedral; see also~\cite{Sanyal13, PS15}.
Building on this result, Kummer~\cite{kummer2021spectral} has shown that the
associated spectral hyperbolicity cones are also spectrahedral.

\subsection*{Categories and Adjointness}
\newcommand\cbodies{\mathcal{K}}%
For a group $G$ acting on a real vector space $V$, let us write $\cbodies(V)^G$
for the class of $G$-invariant convex bodies $K \subset V$. We can interpret
the construction of spectral bodies as a map 
\[
    \Lambda : \cbodies^{\SymGrp_d}(\R^d) \ \to \ \cbodies^{O(d)}(\SymMat) \, .
\]
It follows from Lemma~\ref{lem:sec_project} that the map that takes $A \in
\SymMat$ to $\{ \sigma \lambda(A) : \sigma \in \SymGrp_d \}$ extends to a map
\begin{equation}\label{eqn:O-to-Sym}
    \lambda : \cbodies^{O(d)}(\SymMat) \ \to \ \cbodies^{\SymGrp_d}(\R^d)  
\end{equation}
such that $\lambda \circ \Lambda$ and $\Lambda \circ \lambda$ are the identity
maps. It would be very interesting to see if this can be phrased in
categorical terms that would explain the reminiscence of adjointness of
functors in Proposition~\ref{prop:support}.

\subsection*{Polar convex bodies}
In \cite{BGH13,BGH14} Biliotti, Ghigi, and Heinzner generalized the
construction of Schur-Horn orbitopes to other (real) semisimple Lie groups,
which they called \emph{polar orbitopes}. In particular, they showed that
polar orbitopes are facially exposed and faces are again polar orbitopes.
Kobert~\cite{kobert} gave explicit spectrahedral descriptions of polar
orbitopes involving the fundamental representations of the associated Lie
algebra. It would be interesting to generalize our spectrahedral
representations of spectral polyhedra to this setting.  A first step was taken
in~\cite{BGH16}, where~\eqref{eqn:O-to-Sym} was studied for polar
representations.

\subsection*{Spectral zonotopes}
For $z \in \R^d$, we denote the segment with endpoints $-z$ and $z$ by
$[-z,z]$.
A \Def{zonotope} is a polytope of the form
\[
    Z \ = \ [-z_1, z_1] + [-z_2, z_2] + \cdots + [-z_m, z_m] \, ,
\]
where $z_1,\dots,z_m \in \R^d$ and addition is Minkowski sum. Zonotopes are
important in convex geometry as well as in combinatorics; see, for example,
~\cite{crt,bolker,dcp}. For $z \in \R^d$, we obtain a symmetric zonotope 
\begin{equation}\label{eqn:zonotope}
    Z(z) \ := \ \sum_{\sigma \in \SymGrp_d} \sigma [-z,z]
\end{equation}
and for $z = e_1 - e_2 = (1,-1,0,\dots,0)$, the resulting symmetric zonotope
is $2 (d-2)! \Pi(d-1,d-3,\dots,-(d-3),-(d-1))$ and thus homothetic to the
standard permutahedron $\Pi(1,2,\dots,d)$. For $z = e_1$, we obtain a dilate
of the unit cube $[0,1]^d$.

We define \Def{spectral zonotopes} as convex bodies of the form
\[
    \Spec{Z(z_1)} + \cdots + \Spec{Z(z_m)} \, ,
\]
where $Z(z_i)$ are symmetric zonotopes. This class of convex bodies includes the
Schur-Horn orbitope $\SH((d-1,d-3,\dots,-(d-1)))$ as well as symmetric matrices 
with spectral norm at most one. 
It follows from Corollary~\ref{cor:Minkowski_sum} that spectral zonotopes are
spectral convex bodies and, in particular, spectral zonotopes form a
sub-semigroup (with respect to Minkowski sum) among spectral convex bodies. It
would be very interesting to explore the combinatorial, geometric, and
algebraic properties of spectral zonotopes.

There are a number of remarkable characterizations of zonotopes;
cf.~\cite{bolker}. In particular, zonotopes have a simple characterization in
terms of their support functions: The support function of a zonotope $Z$ as
in~\eqref{eqn:zonotope} is
given by $ \support{Z}{c} = \sum_{i=1}^m |\inner{z_i,c}|$. We obtain the
following characterization for spectral zonotopes.

\begin{cor}
    A convex body $K \subset \SymMat$ is a spectral zonotope if and only if its
    support function is of the form
    \[
        \support{K}{B} \ = \ \sum_{i=1}^m \sum_{\sigma \in \SymGrp_d}
        |\inner{\sigma z_i, \lambda(B)}| \, ,
    \]
    for some $z_1,\dots, z_m \in \R^d$.
\end{cor}
The support function for $Z(e_1-e_2)$ is 
\[
	\support{Z(e_1-e_2)}{c} \ = \ 2(d-2)! \sum_{i < j} |c_i - c_j| \, .
\]
From Proposition~\ref{prop:support}, we infer that the support function of the
(standard) Schur-Horn orbitope is
\begin{equation}\label{eqn:sSH-supp}
    \support{\SH(d-1,\dots,-(d-1))}{B} \ = \ \sum_{i < j} 
	|\lambda(B)_i - \lambda(B)_j | \ = \ \|\mathcal{M}_B\|_* \,.
\end{equation}
Here $\|\cdot\|_*$ is the \Def{nuclear norm}, that is, the sum of the singular
values and, for fixed $B\in \SymMat$,  $\mathcal{M}_B$ is the linear map from
$d\times d$ skew-symmetric matrices to traceless $d\times d$ symmetric
matrices defined by $\mathcal{M}_B(X) = [B,X] = BX-XB$, which has non-zero
singular values $|\lambda(B)_i - \lambda(B)_j|$ for $1\leq i<j\leq d$.  The
$m_1\times m_2$ nuclear norm ball has a spectrahedral representation of size
$2^{\max\{m_1,m_2\}}$~\cite[Theorem 1.2]{SPW15}, and a projected spectrahedral
representation of size $m_1+m_2$. These observations show that
$\polar{\SH(d-1,\dots,-(d-1))} = \{B\;:\; \|\mathcal{M}_B\|_*\leq
1\}$ has a spectrahedral representation of size
$2^{\binom{d+1}{2}-1}$ and a projected spectrahedral representation of size
$d^2-1$.

A convex body $K \subset \R^d$ is a (generalized) \Def{zonoid} if it is the limit (in the
Hausdorff metric) of zonotopes, or, equivalently, if its support function is of
the form
\begin{equation}
	\label{eq:zonoid}
    \support{K}{c} \ = \ \int_{S^{d-1}} |\inner{c,u}| \, d\rho(u) \, ,
\end{equation}
for some (signed) even measure $\rho$; see~\cite[Ch.~3]{Schneider}. It was
hoped that spectral zonotopes are zonoids but this is not the case. Leif
Nauendorf~\cite{Leif} showed that the Schur-Horn orbitopes
$\SH(d-1,\dots,-(d-1))$ are never zonoids for $d \ge 3$.

A convex body $K\subset \R^d$ is a symmetric zonoid if and only if the measure $\rho$ in~\eqref{eq:zonoid} 
is symmetric. We define \Def{spectral zonoids} as those convex bodies with support functions of 
the form
\[ 
\support{K}{B} \ = \ \int_{S^{d-1}}|\inner{\lambda(B),u}|\, d\rho(u)\, ,\]
where $\rho$ is a symmetric even measure. Examples of spectral zonoids include 
the Schatten $p$-norm balls in $\SymMat$ when $p\geq 2$. Further examples of 
spectral zonoids can be found in \cite[Section 5.1]{aubrun2016zonoids} (in the Hermitian setting) 
and~\cite[Section 5]{burgisser2020probabilistic} (in the setting where the singular values of general matrices 
play the role of eigenvalues of symmetric matrices).

\enlargethispage{2cm}

\bibliographystyle{siam} \bibliography{bibliography.bib}

\begin{thebibliography}{10}

\bibitem{aubrun2016zonoids}
{\sc G.~Aubrun and C.~Lancien}, {\em Zonoids and sparsification of quantum
  measurements}, Positivity, 20 (2016), pp.~1--23.

\bibitem{BGLS01}
{\sc H.~H. Bauschke, O.~G\"{u}ler, A.~S. Lewis, and H.~S. Sendov}, {\em
  Hyperbolic polynomials and convex analysis}, Canad. J. Math., 53 (2001),
  pp.~470--488.

\bibitem{crt}
{\sc M.~Beck and R.~Sanyal}, {\em Combinatorial reciprocity theorems}, vol.~195
  of Graduate Studies in Mathematics, American Mathematical Society,
  Providence, RI, 2018.

\bibitem{BenTal-Nemirovski}
{\sc A.~Ben-Tal and A.~Nemirovski}, {\em Lectures on modern convex
  optimization}, MPS/SIAM Series on Optimization, Society for Industrial and
  Applied Mathematics (SIAM), Philadelphia, PA; Mathematical Programming
  Society (MPS), Philadelphia, PA, 2001.

\bibitem{BGH13}
{\sc L.~Biliotti, A.~Ghigi, and P.~Heinzner}, {\em Polar orbitopes}, Comm.
  Anal. Geom., 21 (2013), pp.~579--606.

\bibitem{BGH14}
\leavevmode\vrule height 2pt depth -1.6pt width 23pt, {\em Coadjoint
  orbitopes}, Osaka J. Math., 51 (2014), pp.~935--968.

\bibitem{BGH16}
\leavevmode\vrule height 2pt depth -1.6pt width 23pt, {\em Invariant convex
  sets in polar representations}, Israel J. Math., 213 (2016), pp.~423--441.

\bibitem{BilleraSarangarajan}
{\sc L.~J. Billera and A.~Sarangarajan}, {\em The combinatorics of permutation
  polytopes}, in Formal power series and algebraic combinatorics ({N}ew
  {B}runswick, {NJ}, 1994), vol.~24 of DIMACS Ser. Discrete Math. Theoret.
  Comput. Sci., Amer. Math. Soc., Providence, RI, 1996, pp.~1--23.

\bibitem{bolker}
{\sc E.~D. Bolker}, {\em A class of convex bodies}, Trans. Amer. Math. Soc.,
  145 (1969), pp.~323--345.

\bibitem{branden2014hyperbolicity}
{\sc P.~Br\"{a}nd\'{e}n}, {\em Hyperbolicity cones of elementary symmetric
  polynomials are spectrahedral}, Optim. Lett., 8 (2014), pp.~1773--1782.

\bibitem{burgisser2020probabilistic}
{\sc P.~B\"{u}rgisser and A.~Lerario}, {\em Probabilistic {S}chubert calculus},
  J. Reine Angew. Math., 760 (2020), pp.~1--58.

\bibitem{dcp}
{\sc C.~De~Concini and C.~Procesi}, {\em Topics in hyperplane arrangements,
  polytopes and box-splines}, Universitext, Springer, New York, 2011.

\bibitem{GW}
{\sc R.~Goodman and N.~R. Wallach}, {\em Symmetry, representations, and
  invariants}, vol.~255 of Graduate Texts in Mathematics, Springer, Dordrecht,
  2009.

\bibitem{GPT13}
{\sc J.~Gouveia, P.~A. Parrilo, and R.~R. Thomas}, {\em Lifts of convex sets
  and cone factorizations}, Math. Oper. Res., 38 (2013), pp.~248--264.

\bibitem{HornJohnson}
{\sc R.~A. Horn and C.~R. Johnson}, {\em Matrix analysis}, Cambridge University
  Press, Cambridge, second~ed., 2013.

\bibitem{kobert}
{\sc T.~Kobert}, {\em A spectrahedral representation for polar orbitopes}.
\newblock \url{https://arxiv.org/abs/1611.05658}, November 2016.

\bibitem{kummer2021spectral}
{\sc M.~Kummer}, {\em Spectral linear matrix inequalities}, Adv. Math., 384
  (2021), p.~107749.

\bibitem{LagariasMehta}
{\sc J.~C. Lagarias and H.~Mehta}, {\em Products of binomial coefficients and
  unreduced {F}arey fractions}, Int. J. Number Theory, 12 (2016), pp.~57--91.

\bibitem{lewis1996convex}
{\sc A.~S. Lewis}, {\em Convex analysis on the {H}ermitian matrices}, SIAM J.
  Optim., 6 (1996), pp.~164--177.

\bibitem{Leif}
{\sc L.~Naundorf}, {\em Schur-horn orbitopes and zonoids}, Master's thesis,
  Freie Universit{\"a}t Berlin, 2015.

\bibitem{Sanyal13}
{\sc R.~Sanyal}, {\em On the derivative cones of polyhedral cones}, Adv. Geom.,
  13 (2013), pp.~315--321.

\bibitem{SSS11}
{\sc R.~Sanyal, F.~Sottile, and B.~Sturmfels}, {\em Orbitopes}, Mathematika, 57
  (2011), pp.~275--314.

\bibitem{PS15}
{\sc J.~Saunderson and P.~A. Parrilo}, {\em Polynomial-sized semidefinite
  representations of derivative relaxations of spectrahedral cones}, Math.
  Program., 153 (2015), pp.~309--331.

\bibitem{SPW15}
{\sc J.~Saunderson, P.~A. Parrilo, and A.~S. Willsky}, {\em Semidefinite
  descriptions of the convex hull of rotation matrices}, SIAM J. Optim., 25
  (2015), pp.~1314--1343.

\bibitem{Schneider}
{\sc R.~Schneider}, {\em Convex bodies: the {B}runn-{M}inkowski theory},
  vol.~151 of Encyclopedia of Mathematics and its Applications, Cambridge
  University Press, Cambridge, expanded~ed., 2014.

\bibitem{Vinnikov}
{\sc V.~Vinnikov}, {\em L{MI} representations of convex semialgebraic sets and
  determinantal representations of algebraic hypersurfaces: past, present, and
  future}, in Mathematical methods in systems, optimization, and control,
  vol.~222 of Oper. Theory Adv. Appl., Birkh\"{a}user/Springer Basel AG, Basel,
  2012, pp.~325--349.

\end{thebibliography}

\end{document}